\documentclass[11pt,reqno]{amsart}
\usepackage{amssymb}
\usepackage{mathptmx}
\usepackage{mathtools}
\usepackage{mathrsfs}
\usepackage[export]{adjustbox}
\usepackage[all]{xy}
\usepackage{stmaryrd}
\usepackage{bbm}
\usepackage{fancyhdr}
\usepackage{xcolor}
\usepackage{hyperref}
\usepackage{tikz-cd}
\usepackage{cite}
\usepackage{amsmath,amsfonts}
\usepackage{graphicx}
\usepackage{wrapfig}
\usepackage{array}
\usepackage{amsthm}
\usepackage[left=1.2in, right=1.2in, top=1in, bottom=1in]{geometry}
\usepackage{etoolbox}

\patchcmd{\section}{\scshape}{\bfseries}{}{}
\makeatletter
\renewcommand{\@secnumfont}{\bfseries}
\makeatother
\xyoption{all}
\newcommand{\mycomment}[1]{}

\theoremstyle{definition}
\newtheorem{mydef}{\textbf{Definition}}[section]
\newtheorem{myeg}[mydef]{\textbf{Example}}

\newtheorem{rmk}[mydef]{\textbf{Remark}}
\newtheorem{construction}[mydef]{\textbf{Construction}}

\theoremstyle{plain}

\newtheorem{mythm}[mydef]{\textbf{Theorem}}
\newtheorem{lem}[mydef]{\textbf{Lemma}}

\newtheorem{claim}[mydef]{\textbf{Claim}}

\newtheorem*{que}{\textbf{Question}}

\begin{document}

\title{Nontrivial solutions for homogeneous linear equations over some non-quotient hyperfields }

	\author{David Hobby}
	\address{State University of New York at New Paltz, NY, USA}
	\curraddr{}
	\email{hobbyd@newpaltz.edu}

	\author{Jaiung Jun}
	\address{State University of New York at New Paltz, NY, USA}
	\curraddr{}
	\email{junj@newpaltz.edu}
	
	\makeatletter
	\@namedef{subjclassname@2020}{%
		\textup{2020} Mathematics Subject Classification}
	\makeatother
	
	\date{}
	
	\dedicatory{}

\maketitle

\begin{abstract}
We introduce a class of hyperfields which includes several constructions of non-quotient hyperfields. We then use it to partially answer a question posed by M.~Baker and T.~Zhang:  Does a system of homogeneous linear equations with more unknowns than equations always have a nonzero solution?
We also consider a class of hyperfields that was claimed in the literature to be non-quotient, and show that this is false.
\end{abstract}


\section{Introduction}

Hyperfields are algebraic structures which assume the same axioms as fields except that one allows addition to be \emph{multi-valued}. An idea of multi-valued addition may seem exotic, however, for any given field $F$ and a multiplicative subgroup $G\leq F^\times$, the pair $(F,G)$ naturally defines a hyperfield $F/G$, called a quotient hyperfield (Example \ref{example: quotient}). In fact, Krasner \cite{krasner1956approximation} utilized the notion of quotient hyperfields to define a notion of limits of local fields, and investigated their algebraic extensions and Galois groups in positive characteristic. In doing so, Krasner \cite{krasner1983class } noticed that in order for the theory of hyperfields (or hyperrings) to be robust, one needed to find examples of non-quotient hyperfields. Shortly after, C.G.~Massouros \cite{massouros1985methods, massouros1985theory} and A.~Nakassis \cite{nakassis1988recent} found examples of non-quotient hyperfields. Over time, more hyperfields were proved to be non-quotient.  We refer the reader to \cite{massouros2023on} for the overview.

Recently, hyperfields are being actively studied in connection to many areas of mathematics. For instance, several authors employ hyperfields in number theory and algebraic geometry: A.~Connes and C.~Consani \cite{con3}, O.~Viro \cite{viro2010}, O.~Lorscheid \cite{lorscheid2022tropical}, and J.~Jun \cite{jun2017geometry}. 

One manifestation is that the flexibility of the structure of hyperfields allows one to unify various mathematical objects under the same framework. In fact, in their celebrated work \cite{baker2017matroids}, M.~Baker and N.~Bowler unify various generalizations of matroids by using hyperfields, thereby allowing one to prove theorems concerning all these generalizations of matroids at the same time. 

With the same motivation, in \cite{baker2021descartes}, M.~Baker and  O.~Lorscheid  develop a theory of multiplicities of roots for polynomials over hyperfields to unify both Descartes' rule of signs and Newton's polygon rule. This naturally leads one to the work of M.~Baker and T.~Zhang \cite{baker2022some}, where they unify and provide conceptual proof for various results about ranks of matrices that are scattered throughout the literature by using the framework of matroids over hyperfields. 

With the notions of ranks of matrices over hyperfields, one may ask an elementary yet fundamental question from the classical rank-nullity theorem as follows:

\begin{que}\cite[Question 5.13]{baker2022some}
Does a system of homogeneous linear equations over a hyperfield $H$ with more unknowns than equations always have a nontrivial solution?
\end{que}

One can easily check that if a hyperfield $H$ is a quotient hyperfield, say $H=F/G$, then one can reduce this question to the case of the field $F$, which then directly follows from the rank-nullity theorem.
It naturally follows that one should investigate the question for the remaining hyperfields, the non-quotient ones.

In this paper, we partially answer the above question by showing that there exists an infinite class of non-quotient hyperfields which give a positive answer to the question. 
This class contains the original examples in \cite{massouros1985methods} and \cite{massouros1985theory}, but not the examples from \cite{nakassis1988recent}.
Accordingly, we call them Massouros hyperfields in our Definition \ref{definition: Massouros def}.

We then prove that Massouros hyperfields have the property that ``Fewer Equations Than Variables Implies Nontrivial Solutions'', which we use the acronym FETVINS for.

Our main idea is based on an observation that for Massouros hyperfields, cleverly choosing the values of three terms in a homogeneous linear equation guarantees that it has a nontrivial solution.
This works because the hyperaddition of the three terms will already be the entire hyperfield $F$.
By our Lemma \ref{permanent F lemma}, once this happens adding more terms keeps the sum equal to $F$.
We introduce the notion of hyperfields with ``large sums'' to codify this phenomenon in Definition \ref{definition: larger sums}. 

This essentially lets us reduce a system of homogeneous linear equations to a system with exactly three terms in each equation. Then, we prove that one can find a nontrivial solution in this case.


\section{Preliminaries}

In this section, we recall basic definitions and examples for hyperfields. 

\begin{mydef} \label{hyperaddition definition}
Let $H$ be a nonempty set. By a \emph{hyperaddition}, we mean a function
\begin{equation*}
+:H \times H \to \mathcal{P}^*(H), 
\end{equation*}

where $\mathcal{P}^*(H)$ is the set of nonempty subsets of $H$, satisfying the following conditions:
\begin{enumerate}
    \item
$a+b=b+a$.
\item 
$(a+b)+c=a+(b+c)$.
\end{enumerate}
\end{mydef}

Since sums are usually multi-valued, we adopt the following conventions.
We write ``$x$'' for the one-element set $\{x\}$.  
For sets $A$ and $B$, $A+B$ denotes the union of all the sets $a+b$ for $a \in A$ and $b \in B$.

\begin{mydef} \label{hypergroup definition}
By a \emph{hypergroup}, we mean a set $H$ with  hyperaddition which satisfies the following conditions:
\begin{enumerate}
    \item 
$\exists~!~ 0\in H$ such that $0+h=h$ for all $h \in H$.
\item 
For each $x \in H$, $\exists~!~y \in H$ such that $0 \in x+y$. We denote $y:=-x$.
\item 
If $x \in y+z$, then $z \in (x-y):=x+(-y)$. 
\end{enumerate}
\end{mydef}

\begin{mydef} \label{hyperring definition}
By a \emph{hyperring}, we mean a set $R$ with hyperaddition $+$ and usual multiplication $\cdot$ satisfying the following:
\begin{enumerate}
    \item 
$(H,+)$ is a hypergroup. 
\item 
$(H,\cdot)$ is a multiplicative monoid. 
\item 
$a(b+c)=ab+ac$ for all $a,b,c \in R$. 
\item 
$r\cdot 0 = 0$ for all $r \in R$. 
\end{enumerate}
If $H - \{0\}$ is a multiplicative group, then $H$ is said to be a \emph{hyperfield}.
\end{mydef}

\begin{myeg}[Krasner hyperfield]
Let $\mathbb{K}=\{0,1\}$. We define the following addition:
\[
0+1=1, \quad 0+0=0, \quad 1+1=\{0,1\}.
\]
Multiplication is same as $\mathbb{F}_2$. Then $\mathbb{K}$ is a hyperfield, called the \emph{Krasner hyperfield}. 
\end{myeg}

\begin{myeg}[Sign hyperfield]
Let $\mathbb{S}=\{-1,0,1\}$. We define the following addition:
\[
0+1=1, \quad 0+(-1)=-1,\quad 0+0=0, \quad 1+1=1, \quad (-1)+(-1)=-1, \quad 1+(-1)=\mathbb{S}. 
\]
Multiplication is same as $\mathbb{F}_3$. Then $\mathbb{S}$ is a hyperfield, called the \emph{sign hyperfield}. 
\end{myeg}

\begin{myeg}[Tropical hyperfield]
Let $\mathbb{T}:=\mathbb{R}\cup \{-\infty\}$. Hyperaddition for $\mathbb{T}$ is as follows: for $a,b \in \mathbb{T}$, 
\[
a+b=\begin{cases}
\max\{a,b\} \textrm{ if $a\neq b$,} \\
[-\infty, a] \textrm{  if $a=b$}.
\end{cases}
\]
Multiplication is the usual addition of real numbers with $a\cdot(-\infty)= (-\infty)\cdot a =-\infty$. Then $\mathbb{T}$ is a hyperfield, called the \emph{tropical hyperfield}. 
\end{myeg}

\begin{myeg}[Phase hyperfield]
Let $\mathbb{P}=S^1\cup \{0\}$, where $S^1$ is the complex unit circle. Multiplication is usual multiplication of complex numbers, and addition is defined as follows: for $a, b \in S^1$
\begin{equation}
a+b=\begin{cases}
\{0,a,-a\} \textrm{ if $b=-a$ }, \\
\textrm{all points in the shorter of the two open arcs of $S^1$ connecting $a$ and $b$ if $b\neq a$.}
\end{cases}
\end{equation}
Also $a+0=a$ for all $a \in \mathbb{P}$. Then $\mathbb{P}$ is a hyperfield, called the \emph{phase hyperfield}. 
\end{myeg}

\begin{myeg}[Quotient hyperfield]\label{example: quotient}
In general, let $A$ be a commutative ring and $G \leq A^\times$, the multiplicative group of the units of $A$. Then $G$ naturally acts on $A$ by multiplication and we obtain the set $A/G$ of equivalence classes. Let $[a]$ be the equivalence class of $a \in A$. One defines the following hyperaddition:
\[
[a]+[b]:=\{[c] \mid c=ag_1+bg_2 \textrm{ for some } g_1,g_2 \in G\}.
\]
Multiplication is $[a]\cdot[b]=[ab]$. With these two operations, $A/G$ becomes a hyperring and if $A$ is a field, then $A/G$ is a hyperfield. 
\end{myeg}

\begin{myeg}
Let $G=\{1,-1,i,-i\} \leq \mathbb{C}^\times$, and $R=\mathbb{C}/G$. Then one has $[1]=\{1,-1,i,-i\}=G$, and
\[
[1]+[1]=\{[c] \mid c=a+b, \textrm{ where } a,b \in G\}=\{[0],[1+i] ,[2]\}.
\]
In particular, $[1]=-[1]$.
\end{myeg}


\section{FETVINS hyperfields}

We will partially answer a question asked by M.~Baker and T.~Zhang in \cite{baker2022some}.
To facilitate our discussion, we name the property in question.

\begin{mydef} \label{definition: FETVINS definition}
We say that a hyperfield $F$ has the 
{\em FETVINS} property if
``Fewer Equations Than Variables Implies Nontrivial Solutions''.
That is, whenever we have a system of $k$ homogeneous linear equations in $n$ variables where $k < n$, there are values for the variables $x_1, x_2, \dots, x_n$ which are not all zero, so that all $k$ of the homogeneous linear equations are true.
\end{mydef}

In this terminology, Baker and Zhang ask :  ``Do all hyperfields have FETVINS?'' It is observed in \cite{baker2022some} that since all fields have FETVINS, all quotient hyperfields do.
Thus it is natural to investigate whether or not FETVINS holds in non-quotient hyperfields.
We will establish that FETVINS does hold in various known examples of non-quotient hyperfields, providing evidence that it could be true that all hyperfields had FETVINS.

In a series of papers starting with \cite{massouros1985methods}, C. Massouros and others produced examples of hyperfields which were not produced as quotients of fields.
We call these hyperfields {\em non-quotient}. We refer the reader to a recent survey article \cite{massouros2023on}.
The following class contains the first hyperfields proved to be non-quotient.

\begin{construction}\label{construction: Massouros def}
Consider hyperfields constructed as follows.
Let $G$ be an abelian group written multiplicatively, with identity $1$.
Let $F$ be the set $G \cup {0}$.
Define multiplication on $F$ to be that of $G$, with the added condition that $0 \cdot x = x \cdot 0 = 0$ for all $x \in F$.
Now define the multi-valued addition on $F$ as follows.
\begin{enumerate}
    \item $0 + x = x + 0 = \{x\}$ for all $x$,
    \item $x + x = F - \{x\}$ for all $x \neq 0$, and 
    \item $x + y = y + x = \{x,y\}$ whenever $x$ and $y$ are distinct and nonzero.
\end{enumerate}
We denote the hyperfield $F$ obtained in this way from $G$ by $F_G$.
\end{construction}

C.~Massouros introduced the above class of hyperfields in \cite{massouros1985methods}, and showed there that some hyperfields in the class were non-quotient.

Massouros also looked at similar structures which were slightly more complicated.
For example, the non-quotient hyperfield in the paper \cite{massouros1985theory} has an underlying multiplicative group which is a direct product of a group $G$ and the two element group $\{1,-1\}$, where the negative of $(x,i) \in G \times \{1,-1\}$ is $(x,-i)$. 
Massouros defines the addition of this hyperfield so that the sum of two elements is mostly determined by their components in $G$.
Observing that the multiplicative group has an obvious homomorphism onto the group $G$, we can generalize this construction to include all hyperfields with a similar homomorphism.

All of these first hyperfields proved non-quotient by Massouros are contained in the following class.
It may also contain some quotient hyperfields, but that is irrelevant.
We will show that all hyperfields in the class have FETVINS.

\begin{mydef}\label{definition: Massouros def}
Let $G$ be an abelian group, written multiplicatively. Let $\phi : G \to H$ be a group homomorphism where $|\phi(G)| \geq 3$.
Adjoin $0$ to $G$, and let $F = G \cup \{0\}$.
Define multiplication on $F$ by using the multiplication of $G$ with the additional condition that $0 \cdot x = x \cdot 0 = 0$ for all $x \in F$.
Then the hyperfield $F$ (constructed in Construction \ref{construction: Massouros def}) is {\em Massouros} if and only if its hyperaddition satisfies the following conditions.
\begin{enumerate}
    \item For all $x \in G$, $\phi(-x) = \phi(x)$.
    \item For all $x,y \in G$, $\phi(x) = \phi(y)$ implies $G - \phi^{-1}(\phi(x)) \subseteq x + y$.  
    \item For all $x,y \in G$ with $\phi(x) \neq \phi(y)$, $\phi^{-1}(\{\phi(x),\phi(y)\}) \subseteq x + y$.
\end{enumerate}
\end{mydef}

Here is one of Massouros' original definitions of a class containing non-quotient hyperfields.
(From \cite{massouros1985theory}, or as it is repeated in \cite{massouros2023on}.)
\begin{mydef} \label{definition: Massouros original}
Let $K$ be a multiplicative group with more than two elements, and consider its direct product $L$ with the two element group $\{1,-1\}$.
Adjoin a new element $0$ to $L$, define $0 \cdot 0 = 0$, $0 + 0 = 0$ and define 
$(x,i) \cdot 0 = 0 \cdot (x,i) = 0$ and 
$(x,i) + 0 = 0 + (x,i) = (x,i)$ for all $(x,i)$ in $L$. 
We also define
\begin{enumerate}
    \item $(x,i) + (x,i) = L - \{(x,i),(x,-i)\}$,
    \item $(x,i) + (x,-i) = \{0\} \cup (L - \{(x,i),(x,-i)\})$, and 
    \item $(x,i) + (w,j) = \{(x,i),(x,-i),(w,j),(w,-j)\}$
\end{enumerate}
for all $(x,i), (w,j) \in L$ with $x \neq w$.
\end{mydef}

Here is a small example of a Massouros hyperfield, constructed similarly to Definition \ref{definition: Massouros original}.

\begin{myeg}\label{example: Massouros}
We let $K$ be $\{1,a,a^2\}$ under multiplication, so $a$ is a primitive third root of $1$.
Then we form $L = K \times \{1,-1\}$, but for simplicity we will write it as $\{1,-1,a,-a,a^2,-a^2\}$, where 
$\langle 1,1 \rangle = 1$, 
$\langle 1,-1 \rangle = -1$, 
$\langle a,1 \rangle = a$, 
$\langle a,-1 \rangle = -a$, 
$\langle a^2,1 \rangle = a^2$, and $\langle a^2,-1 \rangle = -a^2$.
For the homomorphism $\phi$, we use the projection of $K \times \{1,-1\}$ onto $K$.
Thus $\phi(1) = \phi(-1) = 1$, 
$\phi(a) = \phi(-a) = a$ and
$\phi(a^2) = \phi(-a^2) = a^2$,
showing that (1) of Definition \ref{definition: Massouros def} is satisfied.

To see (2) of Definition \ref{definition: Massouros def}, note that $\phi((x,i)) = \phi((w,j))$ implies $(w,j) = (x,\pm i)$, and that
\[
L - \phi^{-1}(\phi((x,i))) = L - \{(x,i),(x,-i)\}.
\]

For (3), let $(x,i)$ and $(w,j)$ be given, where 
$\phi((x,i)) \neq \phi((w,j))$.
Then $x \neq w$, and $(x,i) + (w,j) = \{(x,i),(x,-i),(w,j),(w,-j)\} = \phi^{-1}(\{\phi((x,i)),\phi((w,j))\}$.
For example, $\phi(a) = a \neq a^2 = \phi(-a^2)$, and 
$a + -a^2 = \{a,-a,a^2,-a^2\} \subseteq \phi^{-1}(\{a,a^2\} = \phi^{-1}(\{\phi(a),\phi(-a^2)\})$.
\end{myeg}

Our definition of ``large sums'' will allow elements to be distinct from their negations.

\begin{mydef}\label{definition: larger sums}
Let $F=F_G$ be as in Definition \ref{definition: Massouros def}. We say $F$ has {\em large sums} if there is some nontrivial group homomorphism $f : G \to H$ satisfying the following condition:
\[
\forall~x,y,z \in G~ \textrm{ if } f(x) = f(y) \neq f(z), \textrm{ then } x+y+z = F.
\]
\end{mydef}

With this definition, all Massouros hyperfields have large sums.

\begin{lem}
Let $F=(F_G,\phi)$ be a Massouros hyperfield, where $\phi : G \to H$ is a nontrivial group homomorphism. Then $F$ has large sums.    
\end{lem}
\begin{proof}
Let $x,y,z \in G$ with $\phi(x) = \phi(y) \neq \phi(z)$. Since $F$ is Massouros, we have that
\[
G - \phi^{-1}(\phi(x)) \subseteq x+y, \quad \phi(-z) = \phi(z) \neq \phi(x).
\]
It follows that $-z \in x+y$ and $0 \in -z + z \subseteq (x+y) + z$.
Similarly, $\phi(-z) = \phi(z)$ implies that
\[
G - \phi^{-1}(\phi(z)) \subseteq (-z + z) \subseteq x+y+z.
\]
It remains to show that $\phi^{-1}(\phi(z))$ is a subset of $x+y+z$.
So let $w \in \phi^{-1}(\phi(z))$ be given.
Since $|\phi(G)| \geq 3$, there exists $u \in G$ with $\phi(u)$ distinct from $\phi(x)$ and $\phi(z)$.
Then $u \in x+y$, and $w \in u+z \subseteq x+y+z$.    
\end{proof}

We now consider a situation analogous to one in linear algebra.
We often write $Y \ni x$ to mean $x \in Y$.
A statement of the form
$a_1 x_1 + a_2 x_2 + \cdots +a_n x_n \ni 0$
will be called a {\em homogeneous linear equation} with coefficients 
$a_1, a_2, \dots, a_n$.
We adopt the convention that terms with coefficients of zero do not appear, so that $a_1, a_2, 
 \dots, a_n$ are all nonzero in $a_1 x_1 + a_2 x_2 + \cdots + a_n x_n \ni 0$.
 We intend to show that all Massouros hyperfields have FETVINS.

Our plan is to use that Massouros hyperfields have large sums.
Then for a homogeneous linear equation $a_1 x_1 + a_2 x_2 + \cdots + a_n x_n \ni 0$, we can use three variables of our choice to get a sum of $F$.
For example, we pick $x_1$, $x_2$ and $x_3$ so that $\phi(a_1 x_1) = \phi(a_2 x_2) \neq \phi(a_3 x_3)$.
Then $a_1 x_1 + a_2 x_2 + a_3 x_3 = F$ since we have large sums.
The following obvious lemma then implies that we can ignore the remaining $n-3$ variables since the homogeneous linear equation will now be true regardless of their values.

\begin{lem} \label{permanent F lemma}
In any hyperfield $F$, $F + x = F$ for any $x \in F$.   
\end{lem}
\begin{proof}
Let $x$ be given, so we must show $y \in F+x$ for an arbitrary $y \in F$. But, one can easily see  
\[
y = y+0 \subseteq y + ((-x) + x) = (y + (-x)) +x \subseteq F+x.
\]   
\end{proof}

Another way to look at what we are doing is that we are solving ``strengthened'' equations.
\begin{mydef}\label{definition: strengthening}
In the hyperfield $F$, the {\em strengthening} of the equation $a_1 x_1 + a_2 x_2 + \cdots + a_k x_k \ni 0$ is the equation $a_1 x_1 + a_2 x_2 + \cdots + a_k x_k \supseteq F$.
\end{mydef}

\begin{rmk}
   Lemma \ref{permanent F lemma} gives us that any assignment of values to $x_1, x_2, \dots, x_k$ that solves $a_1 x_1 + a_2 x_2 + a_3 x_3 \supseteq F$ also solves $a_1 x_1 + a_2 x_2 + \cdots + a_k x_k \ni 0 $. 
\end{rmk}

Now we can state our main theorem. 

\begin{mythm}\label{theorem: main theorem}
Every Massouros hyperfield has FETVINS.
\end{mythm}
\begin{proof}
Our proof is long, and best presented as a sequence of claims and definitions.

Let us first agree on standard terminology. We will look at hyperfields $F = G \cup \{0\}$ which have large sums using the group homomorphism $\phi : G \to H$.
Then we will consider systems of $e$ homogeneous linear equations in $v$ variables, with nonzero coefficients from $F$.
We must show nontrivial solutions exist whenever $v > e$. 

First observe that it suffices to prove FETVINS in the restricted case where none of the equations in a homogeneous system has exactly one or two variables occurring in it. This is because such equations can be removed from the system while simultaneously removing a variable, so that solving the smaller system produces a nontrivial solution of the original system.

In fact, first consider a  two variable equation, which we write as $ax + by \ni 0$. Picking one of the variables to eliminate, say $y$, we set $y = -b^{-1} ax$, and make this substitution for all occurrences of $y$ in the system. Then $ax + by = ax - ax \ni 0$ and the equation can be eliminated from the system since it will always be true. Note that removing a two variable equation may create additional two variable equations. For example, the three variable equation $cw + dx + ey \ni 0$ reduces to a two variable equation when we substitute $y = -b^{-1} ax$ in it.
This substitution gives $cw + dx + e(-b^{-1} ax) \ni 0$, or $cw + (d - eab^{-1}) x \ni 0$.
While $(d - eab^{-1})$ need not be a single element, that is not a problem.
We replace $(d - eab^{-1})$ in the equation by any of its nonzero elements. (Nonzero elements exist,  since for any Massouros hyperfield $F$ and nonzero elements $x,y \in F$, we have $|x+y|\geq 2$.) Any solution of this modified equation will also solve the original.

Now, we may assume that all two variable equations have already been removed from the system. If we now have a one variable equation $ax \ni 0$, we can set $x = 0$ throughout the system.  
This makes the equation $ax \ni 0$ true, so we remove it from the system.
This produces a new system with fewer equations than variables, where nontrivial solutions of the new system produce nontrivial solutions of the old system.
As before, ``zeroing out'' a variable in this way can produce additional one variable or two variable equations.
We continue the process until no one variable or two variable equations remain. Since the resulting system has more variables than equations, it is not vacuous.

We can generalize the argument for one variable equations in the following way.

\begin{mydef}\label{definition: subsystem and pile}
A {\em subsystem} of a system of homogeneous linear equations consists of some subset of the equations, together with all of the variables occurring in those equations. Call a subsystem of our system of equations a {\em pile} if it has at least as many equations as variables.   
\end{mydef}

If our system contains a pile, we may set all of the variables in the pile equal to zero and produce a smaller system that has more variables than equations.
A solution of the smaller system then extends to a solution for the original system.
Thus we may reduce the proof to showing that systems with no piles (which we call {\em pilefree}) have nontrivial solutions.

Now that we are only dealing with systems where all equations have $3$ or more variables, we replace the original problem of solving a system of homogeneous linear equations with the problem of solving the system {\em strengthened system}, made by strengthening all those equations (Definition \ref{definition: strengthening}). Since any solution of the following equation
\[
a_1 x_1 + a_2 x_2 + \cdots +a_k x_k \supseteq F
\]
is a solution of the following equation
\[
a_1 x_1 + a_2 x_2 + \cdots +a_k x_k \ni 0,
\]
this is valid. 

We may also assume that all equations have exactly three variables.
Our solution process will use the fact that the hyperfield has large sums so that we can ignore all variables in an equation after the first three that we assign values to.
This means that we can replace an equation with more than three variables by any equation made by reducing it to the sum of three terms.
Solving the system with these reduced equations will also produce a solution to the original system.
The one issue that might arise in passing to a system consisting only of three variable equations is that doing so might produce piles. To resolve this, we prove the following.

\begin{claim}\label{3-variable claim}
Consider a pilefree system with more variables than equations where all equations have three or more variables.
Then all equations with more than three terms may be reduced to three variable equations by removing terms, in such a way that the resulting system is pilefree, consists only of three variable equations, and where any solution of the reduced system is a solution of the original.
\end{claim}
\begin{proof}
Given a pilefree system where all equations have three or more variables, we define its {\em excess} to be the sum (over all equations) of the number of variables more than three in that equation.
We prove the claim by induction on the excess of a system.
The basis is the case where all equations have exactly three variables, no reduction is needed in this case.
Now consider a pilefree system where the equation $E_1$ is 
\begin{equation}\label{eq: E_1}
  a_1 x_1 + a_2 x_2 + a_3 x_3 + \cdots + a_k x_k, \quad \textrm{with $k>3$}.  
\end{equation}
Consider removing the term $a_1 x_1$ from $E_1$.
If doing this produced a pilefree system, we would be done.
So assume it does not produce a pilefree system. Then in the new system, there is a pile with 
\[
a_2 x_2 + a_3 x_3 + \cdots + a_k x_k
\]
as one of its equations. Call this pile $P_1$, and note that the variables $x_2, x_3, \dots, x_k$ are all in $P_1$, and that $x_1$ does not appear in $P_1$ since the original system is pilefree.
Let $v_1$ be the number of variables in $P_1$.
Since $P_1$ is a pile, there must be at least $v_1$ equations on the variables in $P_1$ in the new system.
The original system is the same, except it had $E_1$ instead of $a_2 x_2 + a_3 x_3 + \cdots + a_k x_k$, and thus must have at least $v_1 - 1$ many equations on the variables in $P_1$.

Now repeat the argument, considering the removal of the term $a_2 x_2$ from $E_1$.
We may again assume that this produces a pile in the new system, since otherwise we would be done.
This gives us a corresponding pile $P_2$ that would be made by removing $a_2 x_2$, where there are $v_2$ variables in $P_2$ and at least $v_2 - 1$ equations in those variables in the original system.

Let $v_{12}$ be the number of variables that appear in both $P_1$ and $P_2$.
Since our original system is pilefree, it has at most $v_{12} - 1$ equations that only use those variables.
Now consider $U$, the union of the sets of variables in $P_1$ and $P_2$.
The set $U$ contains $v_1 + v_2 - v_{12}$ variables, and includes all the variables appearing in $E_1$.

Counting the equations on variables in $U$ in the original system, we have $E_1$, the $e_1 \geq v_1 - 1$ or more equations in $P_1$, and the $e_2 \geq v_2 - 1$ or more equations in $P_2$.  This  double-counts the equations in variables that are in both $P_1$ and $P_2$, but there are at most $v_{12} - 1$ such double-counted equations.
This gives us at least 
\begin{equation}
1 + e_1 + e_2 - (v_{12} - 1) \geq 1 +(v_1 - 1) + (v_2 - 1) - (v_{12} - 1) = v_1 + v_2 - v_{12} 
\end{equation}
equations on $U$, giving us a pile in the original system.
This contradiction completes the proof of the claim. Finally, any solution of the reduced system is a solution of the original system since the hyperfield has large sums we we noted above.
\end{proof}

\begin{claim}\label{claim: second claim}
Consider a pilefree system of homogeneous linear equations with more variables than equations, where all equations have exactly $3$ variables.
Then this system has a solution where all of the variables are assigned nonzero values.    
\end{claim}
\begin{proof}
We prove the stronger claim, that every pilefree system of {\em strengthened} three variable equations with more variables than equations has a solution with all variables nonzero.
To do this, we proceed by induction on the number of variables in the system.

Our basis is the smallest non-vacuous case, where the system has three variables, call them $x$, $y$ and $z$.
Then there are two or fewer 3-variable equations on these three variables.
We will deal with the situation where there are two 3-variable equations, since the others can be reduced to it by adding equations.
Let the two equations be
\begin{equation}\label{eq: system}
ax + by + cz \supseteq F \quad \textrm{and}\quad dx + ey + fz \supseteq F,
\end{equation}
where $a,b,c,d,e$ and $f$ are nonzero. Take $x$ and $y$ to be nonzero elements of F with
\[
\phi(ax) = \phi(by).
\]
First, suppose that $\phi(dx) = \phi(ey)$. Then, any $z$ satisfying the following conditions
\begin{equation}\label{eq: conditions}
\phi(cz) \neq \phi(ax) \quad \textrm{and} \quad \phi(fz) \neq \phi(dx)
\end{equation}
yields a solution to the system \eqref{eq: system}. Since $|\phi(G)| > 2$, a nonzero value for $z$ can be found to satisfy \eqref{eq: conditions}.

Now, consider the case where $\phi(dx) \neq \phi(ey)$. The second equation is solved if $\phi(fz) = \phi(dx)$ or $\phi(fz) = \phi(ey)$, that is if 
$\phi(z) \in \{\phi(f^{-1}dx),\phi(f^{-1}ey\}$.
This gives us two values of $\phi(z)$ to choose from, so for at least one of these values of $\phi(z)$, we also have $\phi(cz) = \phi(c)\phi(z)\neq \phi(ax)$, giving a solution of both equations, and proving the basis case.

Now for the induction step.
Suppose we have a system ${\mathcal S}$ on $v+1$ many variables, where the claim holds for all systems on $v$ variables.
Let $V$ be the set of all variables in this system. 
Since ${\mathcal S}$ has $v$ or fewer equations on $3$ variables, the Pigeonhole Principle gives us that there is some variable $z$ which is in two or fewer $3$-variable equations.
We may as well assume there are two $3$-variable equations involving $z$, and write them as
\begin{equation}\label{eq: subtract}
au + bw + cz \supseteq F \quad \textrm{and} \quad dx + ey + fz \supseteq F.
\end{equation}
We are assuming that $u$, $w$ and $z$ are distinct variables, as are $x$, $y$ and $z$, although $\{u,w\}$ and $\{x,y\}$ may overlap.
Multiplying $dx + ey + fz \supseteq F$ by $c f^{-1}$, we produce the equivalent equation 
\begin{equation}
d'x + e'y + cz \supseteq F,
\end{equation}
where $d' = c f^{-1} d$ and $e' = c f^{-1} e$.

Let $E$ be the following equation
\begin{equation} \label{eq: E}
E: au + bw + d'x + e'y \supseteq F.
\end{equation}
We let ${\mathcal S'}$ be the system on the set of variables $V - \{z\}$ with equations those of ${\mathcal S}$ with the two equations in \eqref{eq: subtract} omitted and $E$ added. We claim that ${\mathcal S'}$ is pilefree. Indeed, suppose that $P'$ is a pile in $\mathcal{S}'$. Then, $P'$ should contain the equation $E$ since otherwise $P'$ is a pile in $\mathcal{S}$. Now, we can produce a pile in ${\mathcal S}$ as follows. Let $P$ be the subsystem made by adding $z$ to the variables of $P'$, removing the equation $E$, and adding the two equations in \eqref{eq: subtract}. This has one more variable and one more equation than $P'$ does, and is a pile if $P'$ is. Hence, $\mathcal{S}'$ is pilefree as claimed.

The system ${\mathcal S'}$ has a $4$-variable equation, but we may argue, as in Claim \ref{3-variable claim}, to produce a system with all $3$-variable equations, and then use the solution to that system to solve ${\mathcal S'}$.
So by the induction hypothesis, ${\mathcal S'}$ has a solution with all of its variables nonzero.

We will find a value of $z$ that gives a nonzero solution of ${\mathcal S}$. First, consider the following:
\begin{equation}\label{eq: terms}
\phi(au), \quad \phi(bw), \quad \phi(d'x), \quad \textrm{ and} \quad \phi(e'y).
\end{equation}
There are two cases. 

\underline{Case 1:} Suppose first that not all of the terms in \eqref{eq: terms} are distinct. If two terms in the same equation have equal images, we may assume that these terms are $au$ and $bw$.
Then to have a solution of $au + bw + cz \supseteq F$, we only need to pick $z$ so that $\phi(cz) \neq \phi(au)$.
This brings us to two subcases.
If $\phi(d'x) = \phi(e'y)$, we also need to pick $z$ so that $\phi(cz) \neq \phi(d'x)$, but this can be done since $|\phi(G)| > 2$.
If $\phi(d'x) \neq \phi(e'y)$, we argue as after \eqref{eq: conditions}.
We need to pick $z$ so that $\phi(cz) = \phi(d'x)$ or $\phi(cz) = \phi(e'y)$;  this can be done since $\phi(cz) \neq \phi(au)$ in at least one case.

Now assume that no terms in the same equation have equal images.
Since the images of the four terms are not distinct, we must have terms in different equations with equal images.
We may assume that $\phi(au) = \phi(d'x)$.
Now we take $z$ so that $\phi(cz) = \phi(au) = \phi(d'x)$.
This solves $au + bw + cz \supseteq F$, since $\phi(au) = \phi(cz)$ and $\phi(bw) \neq \phi(au)$.
Similarly, we have $\phi(d'x) = \phi(cz)$ by our choice of $z$, and $\phi(d'x) \neq \phi(e'y)$ since $d'x$ and $e'y$ appear in the same equation.
Thus $d'x + e'y + cz \supseteq F$. 

\underline{Case 2:} The second case is where $\phi(au)$, $\phi(bw)$, $\phi(d'x)$ and $\phi(e'y)$ are all different.
Since $(au + bw) + (d'x + e'y) \supseteq F$, there is an element $s$ of $au + bw$ where $-s \in (d'x + e'y)$.
If $s = 0$, we would have $au = -bw$, giving $\phi(au) = \phi(bw)$, a contradiction.
Now we take $z$ nonzero with $s = cz$.
This gives us
\[
s + (-s) \in cz + d'x + e'y \quad \textrm{and}\quad  s + s \in cz + au + bw.
\]
These facts are similar enough that we can use them to solve the two strengthened equations, writing $\pm s$ and only using that $\phi(s) = \phi(\pm s)$.

We demonstrate the process using $s + \pm s \in (au + bw) + cz$.  We have $\phi(\pm s) = \phi(s)$, so 
\[
G - \phi^{-1}(\phi(s)) \subseteq au + bw + cz.
\]
Now at least one of $\phi(au)$ and $\phi(bw)$ is distinct from $\phi(\pm s)$, suppose that $\phi(au) \neq \phi(\pm s)$. Then, we have
\[
\phi^{-1}(\phi(\pm s)) \subseteq au + \pm s \subseteq (au + bw) + cz.
\]
Together, these show $au + bw + cz = F$.
The proof that $d' x + e' y + cz = F$ is similar.

This completes our inductive proof of the claim.
\end{proof}

We may now complete the proof of Theorem \ref{theorem: main theorem} as follows.
Let $F$ be a Massouros hyperfield and $\mathcal{S}$ be a homogeneous system of linear equations over $F$. 
\begin{enumerate}
\item We may assume that none of the equations in $\mathcal{S}$ has exactly one or two variables occurring in it. 
\item We can reduce to the case where $\mathcal{S}$ is pilefree (Definition \ref{definition: subsystem and pile}) and all equations in $\mathcal{S}$ have exactly three variables by Claim \ref{3-variable claim}. 
\item Then, the theorem follows from Claim \ref{claim: second claim}.
\end{enumerate}
\end{proof}

Here is an example where we solve a system, illustrating many of the ideas in the proof.

\begin{myeg}
We will work in the Massouros hyperfield from Example \ref{example: Massouros}, so we have
\[
G = \{1,-1,a,-a,a^2,-a^2\},
\]
where $\phi(1)=\phi(-1)=1$,  $\phi(a)=\phi(-a)=a$ and  $\phi(a^2)=\phi(-a^2)=a^2$.

Consider the system:
\begin{align*}
&ar &+ &a^2s &+ &u &+ &v &- &w &+ &x &+ &y &+ &z& &\ni 0\\
&r &- &s  & &   &   + &v &-&aw  &  & & &   &+&az& &\ni 0\\
&ar &- &s &+ &au &+ &av &+ &w & & &+ &ay & & & &\ni 0\\
& & & & & & &av &+ &a^2w & & & & &+ &az& &\ni 0\\
& & & & & & & & & & &ax &+ &a^2y &+ &z& &\ni 0\\
& & & & & & & & & & &x &- &a^2y &+ &az& &\ni 0\\
& & & & & & & & & & &x &+ &ay &- &z& &\ni 0\\
\end{align*}
The last three equations form a pile, so we will set $x$, $y$ and $z$ equal to zero, yielding the simpler system:
\begin{align*}
&ar &+ &a^2s &+ &u &+ &v &- &w & &\ni 0\\
&r &- &s  & &   &   + &v &-&aw  &   &\ni 0\\
&ar &- &s &+ &au &+ &av &+ &w &  &\ni 0\\
& & & & & & &av &+ &a^2w &  &\ni 0\\
\end{align*}
The fourth equation is the two-variable equation $av+a^2w \ni 0$.
We solve this by setting $av = -a^2w$, giving $w = -a^2v$.  
Substituting this in for $w$, we obtain:
\begin{align*}
&ar &+ &a^2s &+ &u &+ &v  & &\ni 0\\
&r &- &s  & &   &   + &av   &   &\ni 0\\
&ar &- &s &+ &au &+ &av  &  &\ni 0\\
\end{align*}
Here, we have replaced $(1+a^2)v$ with $v$, since $1 \in 1+a^2$.
Similarly, we replaced $(1+1)v$ with $av$ and $(a-a^2)v$ with $av$.
Any solution of this system will solve the previous one.

The first and third equations have more than three variables, so we reduce them to three-variable equations by removing the terms $+a^2s$ and $+av$, respectively.
(We can remove terms arbitrarily, as long as we do not create a pile.)
We now have the system:
\begin{align*}
&ar & & &+ &u &+ &v  & &\supseteq F\\
&r &- &s  & &   &   + &av   &   &\supseteq F\\
&ar &- &s &+ &au & &  &  &\supseteq F\\
\end{align*}
We are now using strengthened equations.
There must be a variable appearing in two or fewer equations.
We pick one such variable, say $v$.
The induction step combined the equations containing this variable, producing a smaller system.
In our case, the first two equations are $ar+u+v \supseteq F$ and $r-s+av \supseteq F$.
We multiply the second equation by $a^2$, giving the equation $a^2r - a^2s +v \supseteq F$ which also ends in $v$.
Now we add the two equations and remove $v$, giving 
the equation 
\begin{equation}\label{eq: EE}
E:~ar+u+a^2r-a^2s \supseteq F,
\end{equation}
which was also called $E$ in \eqref{eq: E}.
In this case, $r$ appears twice in the equation so we replace $(a+a^2)r$ with $ar$, giving 
$ar-a^2s+u \supseteq F$.
This gives us the reduced system:
\begin{align*}
&ar &- &a^2s &+ &u & &\supseteq F\\
&ar &- &s &+ &au &   &\supseteq F\\
\end{align*}
This is our basis case.
We focus on the first equation.
We pick nonzero values of r and s that make $\phi(ar) = \phi(-a^2s)$, say $r=1$ and $s = a^2$.
In the second equation, we then have $\phi(ar) = a$
and $\phi(-s) = a^2$, putting us in the case where these two terms are mapped to different elements by $\phi$.
In this case, we obtain a solution by choosing $u$ so that the term $u$ in the first equation has $\phi(u) \neq \phi(ar) = \phi(-a^2s) = a$, solving the first equation.
To solve the second equation, we also need that $u$ is such that the term $au$ has $\phi(au)$ equal to one of 
$\phi(ar) = a$ and $\phi(-s) = a^2$.
Picking $u$ so that $\phi(au) = a^2$ would make 
$\phi(u) = a$, which does not solve the first equation.
But that is not a problem, we pick $u$ so that 
$\phi(au) = a$, say $u = -1$, solving the second equation.
Since this choice has $\phi(u) = 1 \neq a$, it also solves the first equation.
Thus setting $r=1$, $s=a^2$ and $u=-1$ gives a solution of the basis system.

Now we return to the two equations we combined into the equation $E$ in \eqref{eq: EE} in the induction step.
Keeping them in the form where they both have the term $v$, we have:
\begin{align*}
&ar & & &+ &u &+ &v  & &\supseteq F\\
&a^2r &- &a^2s  & &   &   + &v   &   &\supseteq F\\
\end{align*}
The argument at this stage had several cases.
We have $r=1$, $s=a^2$ and $u=-1$, giving a case where
$\phi(ar) = a$, $\phi(u) = 1$, $\phi(a^2r) = a^2$ and 
$\phi(-a^2s) = a$.
Here, both equations have a term that $\phi$ takes to $a$, so we choose $v$ so that $\phi(v) = a$ as well, say $v = a$.
In the first equation we have $\phi(ar) = \phi(v) \neq \phi(u)$, which solves it.
The second equation has $\phi(a^2r) \neq \phi(-a^2s) = \phi(v)$, and is also solved.

Completing the solution, we have $w = -a^2v = -1$, and 
$x = y = z = 0$.
\end{myeg}

\section{Further remarks} \label{section: Further remarks}

The original impetus for this paper was to investigate whether or not FETVINS held in hyperfields known to be non-quotient.
If it did, this would be evidence that FETVINS held in all hyperfields, since we already know it holds in quotient hyperfields.

The paper \cite{massouros2023on} presents a good summary of known non-quotient hyperfields.
It does unfortunately have an error in its Theorem 12.
This concerns modifications of a version of tropical hyperfields.
We start with the following construction.

\begin{mydef} \label{def ordered hyperfield}
Let $\langle E,\cdot \rangle$ be a totally ordered multiplicative group, with an order compatible with its multiplication.
Adjoin a least element $0$, and extend $\cdot$ to $E \cup \{0\}$ by defining $0 \cdot x = x \cdot 0 = 0$ for all $x$.
We have two possible ways to define the hyperaddition on $E \cup \{0\}$.

Using the following definition gives us what we call a {\em closed-ordered} hyperfield.
\begin{equation}
x+y = \begin{cases}
    \max\{x,y\} & \textrm{if}\quad x \neq y\\
    \{z \mid z \leq x\} & \textrm{if}\quad x = y
\end{cases}
\end{equation}

While the following slight modification gives us what we call an {\em open-ordered} hyperfield.
\begin{equation}
x+y = \begin{cases}
    \max\{x,y\} & \textrm{if}\quad x \neq y\\
    \{z \mid z < x\} & \textrm{if}\quad x = y
\end{cases}
\end{equation}
\end{mydef}

Closed-ordered hyperfields are well known, and were introduced by J. Mittas in the early paper \cite{mittas1973}.
Many of them are quotient hyperfields, isomorphic to those obtained using valuations on fields.

The similar open-ordered hyperfields seem to be an obvious modification.
They are mentioned by O. Viro in \cite{viro2010}, where the addition for closed-ordered hyperfields is said to give a ``linear order multigroup'', and the addition for open-ordered hyperfields gives a ``strict linear order multigroup''.

Both types of hyperfields are also discussed by C. Massouros and G. Massouros in \cite{massouros2023on}.
They claim there that all open-ordered hyperfields are non-quotient.
This is false, and the published proof is incorrect.

The argument there is a proof by contradiction, it assumes that 
$\langle E,+,\cdot \rangle$ is an open-ordered hyperfield, isomorphic to a quotient of a field $\langle R, +, \cdot \rangle$ by a multiplicative subgroup $Q$ of $R - {0}$.
Then referring to $1,2,3, \dots$ in $R$, we have $1 \in R$ and $2 = 1+1 \in Q + Q$.
Since $E$ is open-ordered, $Q$ and $Q+Q$ are disjoint, and thus $2Q$ is a distinct class from $Q$.
Continuing, $2Q < Q$, and calculating in $E$ we get $2Q + Q = Q$, giving
$3 = 2+1 \in Q$.
We also have $4 \in 2Q + 2Q$, so $4Q < 2Q$ and the classes $4Q$, $2Q$ and $Q$ are distinct.
The flaw in the proof is a calculation that $7 \notin Q$.
In fact, $7 = 4+3 \in 4Q + 3Q = 4Q + Q = Q$, and there is no contradiction.

One may continue the proof's attempted classification of the classes $nQ$ for $n \in {\mathbb N}$.
Doing so, one obtains that
$Q = 3Q = 5Q = 7Q, \dots$, that
$2Q = 6Q = 10Q = 14Q \dots$, that
$4Q = 12Q = 20Q = \dots$, and so on.
There is no contradiction here, it turns out these can be obtained using the dyadic valuation on the field of rational numbers.

\begin{myeg}
Let ${\mathbb Q}$ be the field of rational numbers, and let 
$P = \{ a/b \colon a \;\textrm{and}\; b\; \textrm{are odd} \}$.
Then $P$ is a multiplicative subgroup, and we can form the quotient hyperfield ${\mathbb Q}/P$.
This quotient has classes $2^nP$ for each $n \in {\mathbb Z}$, with
\[
\cdots < 4P < 2P < P < (1/2)P < \cdots.
\]
We have that $P+P$ is the set of sums
$a/b + c/d = (ad + bc)/cd$ for $a$,$b$,$c$ and $d$ all odd.
But $ad+bc$ will be even, so all of $P+P$ is less than $P$, and
$P+P = \{\dots, 8P,~4P,~2P\}$.
Multiplying through by $2^n$, we get
\[
2^nP+2^nP = \{\dots ,2^{n+3}P,~2^{n+2}P,~2^{n+1}P\} 
\]
for any integer $n$.  This shows that ${\mathbb Q}/P$ is open-ordered.
\end{myeg}

There seem to be other open-ordered hyperfields which are quotients.
It looks like one can take transcendental extensions of ${\mathbb Q}$, extend the dyadic valuation to them non-trivially, and produce other quotient hyperfields with are open-ordered.

The above seems to remove the class of open-ordered hyperfields from serious consideration as a source of possible examples of hyperfields where FETVINS fails.

The remaining class know to contain infinitely many non-quotient hyperfields is given by A. Nakassis in \cite{nakassis1988recent}.
It consists of hyperfields of the following form.

\begin{mydef} \label{Nakassis hyperfield}
Let $G$ be a group with more than $3$ elements, written multiplicatively.
Add a new element $0$, and let $F = G \cup \{0\}$, where $x \cdot 0 = 0 \cdot x = 0$ for all $x \in G$.
Now define the hyperaddition $+$ on $F$ as follows.

\begin{align*}
 x+y &= F - \{0,x,y\} & &\textrm{for all distinct}\; x,y \in G\\
 x+x &= \{0,x\} & &\textrm{for all}\; x \in G\\
 x+0 &= 0+x = x & &\textrm{for all}\; x \in F
\end{align*}
\end{mydef}

It is plausible that a proof similar to that of our Theorem \ref{theorem: main theorem} could show that all or almost all of the above hyperfields have FETVINS.
We have an analogue of large sums for almost all of these hyperfields.
\begin{mythm}
Suppose that $F$ is a hyperfield as in Definition \ref{Nakassis hyperfield}, where $|F| \geq 6$.
Let $x$, $y$ and $z$ be three distinct nonzero elements of $F$.
Then $x + y + z = F$.
\end{mythm}

\begin{proof}
Since $|F| \geq 6$, there are two additional nonzero $v,w \in F$ where 
$v$, $w$, $x$, $y$ and $z$ are all distinct.
Now $x+y = F - \{0,x,y\}$, so we have 
$v,w,z \in x+y$.
Then \[
x+y + z \supseteq \{v,w,z\} +z =
(F-\{0,v,z\})\cup(F-\{0,w,z\})\cup\{0,z\} = (F-\{0,z\})\cup\{0,z\} = F
\]
\end{proof}

\bibliography{rank}\bibliographystyle{alpha}
\end{document}